\documentclass[12pt]{amsart}
\usepackage{fancyhdr}
\usepackage{stmaryrd}
\usepackage{calrsfs}

\usepackage{thmtools,thm-restate}
\usepackage{amsmath}
\usepackage[nospace,noadjust]{cite}
\usepackage{amsfonts}
\usepackage{amssymb,enumerate}
\usepackage{amsthm}
\usepackage{cite}
\usepackage{comment}
\usepackage{color}
\usepackage[all]{xy}
\usepackage{hyperref}
\usepackage{lineno}
\usepackage{stmaryrd}

\usepackage{mathtools}
\usepackage{bm}
\usepackage{graphicx}
\usepackage{psfrag}
\usepackage{url}

\usepackage{fancyhdr}
\usepackage{tikz-cd}
\usepackage{tikz}

\usepackage{dirtytalk}
\usepackage{float}
\usepackage{mathrsfs}
\usepackage{enumitem}

\newtheorem*{maintheorem*}{Main Theorem}
\newtheorem{theorem}{Theorem}[section]
\newtheorem{prop}[theorem]{Proposition}

\newtheorem{lemma}[theorem]{Lemma}
\newtheorem{cor}[theorem]{Corollary}

\theoremstyle{definition}
\newtheorem{definition}[theorem]{Definition}

\numberwithin{equation}{section}

\newcommand{\defn}[1]{{\color{blue!50!black}\emph{#1}}}

\newcommand\mydots{\hbox to 1em{.\hss.\hss.}}

\newcommand{\bubcov}{\lessdot_{\mathsf{bub}}}

\newcommand{\bubleq}{\leq_{\mathsf{bub}}}

\newcommand{\shufleq}{\leq_{\mathsf{shuf}}}

\newcommand{\sindel}{\hookrightarrow}
\newcommand{\transpose}{\Rightarrow}
\newcommand{\Poset}{\mathbf{P}}
\newcommand{\Qoset}{\mathbf{Q}}

\newcommand{\indeg}{\mathsf{in}}
\newcommand{\indeld}{\mathsf{in_{\sindel}}}
\newcommand{\transd}{\mathsf{in_{\transpose}}}
\DeclareMathOperator{\rk}{rk} 

\newcommand{\ShufPoset}{\textbf{\textsf{Shuf}}}
\newcommand{\Bub}{\textbf{\textsf{Bub}}}

\newcommand{\ch}{\widetilde{\operatorname{ch}}}

\def\u{\mathbf{u}}\def\v{\mathbf{v}}\def\w{\mathbf{w}}\def\x{\mathbf{x}}\def\y{\mathbf{y}}

\def\rd{\textcolor{red}}
\def\bl{\textcolor{blue}}

\newcommand{\defs}{\overset{\mathsf{def}}{=}}

\hypersetup{
	pdftitle={FD},
	colorlinks=true,
	linkcolor=blue,
	citecolor=cyan,
	urlcolor=wine
}


\begin{document}

	\mbox{}
	\title{On the Bivariate Characteristic Polynomial of the Shuffle Lattice}
	\author{Annabel Ma} \address{Department of Mathematics, Harvard University, Cambridge, MA 02138} \email{annabelma@college.harvard.edu}

\maketitle

\begin{abstract}
The shuffle lattice was introduced by Greene in 1988 as an idealized model for DNA mutation, when he revealed remarkable combinatorial properties of this structure. In this paper, we prove an explicit formula for the $M$-triangle of the shuffle lattice, a bivariate refinement of the characteristic polynomial, as conjectured by McConville and M\"uhle in 2022, and find a relation between the $M$-triangle and the $H$-triangle, a bivariate refinement of the rank generating function. 
\end{abstract}

\bigskip

\section{Introduction}

Let $\x = x_1\dots x_m$ and $\y = y_1\dots y_n$ be words of length $m$ and $n,$ respectively, where all $m+n$ letters are distinct. Consider all ways to transform $\x$ into $\y$ using a mutation called an \defn{indel} $\u\rightarrow \v,$ where the word $\v$ is reached by deleting some $x_i$ or inserting some $y_j$ into $\u.$ Each intermediate word that appears between $\x$ and $\y$ via some sequence of indel mutations is called a \defn{shuffle word.} The \defn{shuffle lattice} is the poset of intermediate words, where the order relation is defined by the indel operation. 

In 1988, Greene introduced the shuffle poset as an idealized model for DNA mutation, when he discovered several relationships between different combinatorial invariants such as its characteristic polynomial, its zeta polynomial, and its rank generating function~\cite{greene1988posets}. Since then, shuffle lattices have been studied extensively by Doran in 1994~\cite{DORAN1994118}, by Simion and Stanley in 1999~\cite{simion1999flag}, by Hersh in 2002~\cite{hersh2002two}, and by M\"uhle in 2022~\cite{MUHLE2022103521}. 

Then, in 2022, McConville and M\"uhle introduced bivariate refinements of the characteristic polynomial and the rank generating function of the shuffle lattice, called the \defn{$M$-triangle} and the \defn{$H$-triangle}, respectively~\cite{bubblelattices2}. The authors proved a formula for the $H$-triangle and conjectured a formula for the $M$-triangle, which we prove in this paper. 
\begin{restatable}{theorem}{mexplicit}\label{thm:m-explicit-formula}
    For $m,n\geq 0$, the $M$-triangle of the shuffle lattice $\ShufPoset(m,n)$ is given by
\begin{align*}
		M_{m,n}(q,t) 
		& = \sum_{a\geq 0} \binom{m}{a}\binom{n}{a}t^{a}(1-t)^{a}(q-1)^{a}(qt-t+1)^{m+n-2a}.
	\end{align*}
\end{restatable}
To prove this theorem, we find the rational function representation of the series $\sum_{m \geq 0} \sum_{n \geq 0} x^my^nM_{m,n}(q,t),$ which we state in the following theorem:
\begin{restatable}{theorem}{mtrianglerational}\label{thm:m-rational-function}
    We have 
    \begin{displaymath}
        \sum_{m, n \geq 0} x^my^nM_{m,n}(q,t) = \frac{1}{(1-x(qt - t + 1))(1-y(qt - t +1)) - t(1-t)(q-1)xy}.
    \end{displaymath}
\end{restatable}

Then, using the explicit form of the $M$-triangle, we resolve another conjecture by McConville and M\"uhle on an intriguing relationship between the $M$-triangle and $H$-triangle. 
\begin{restatable}{cor}{fromhtom}\label{cor:h-to-m}
    For $m, n \geq 0,$ we have 
    \begin{equation*}\label{eq:h_to_m}
		M_{m,n}(q,t) = (1-t)^{m+n}H_{m,n}\left(\frac{t(q-1)}{1-t},\frac{q}{q-1}\right).
	\end{equation*}
\end{restatable}

\section{Preliminaries}

For integers $m, n \geq 0,$ consider the two disjoint sets of letters 
\begin{displaymath}
	X = \{x_{1},\ldots,x_{m}\}\quad\text{and}\quad Y=\{y_{1},\ldots,y_{n}\}.
\end{displaymath}

\begin{definition}
    A \defn{shuffle word} is an \emph{order-preserving, simple} word over $X \cup Y,$ which is a word using letters from $X$ or $Y$ without duplicates, such that if $x_i$ and $x_j$ both appear in the word, $x_i$ appears before $x_j$ and if $y_i$ and $y_j$ both appear in the word, $y_i$ appears before $y_j$ whenever $i < j.$ Let $\ShufPoset(m,n)$ be the set of all shuffle words.
\end{definition}
For example, the word $\bl{y_1y_2}\rd{x_2}\bl{y_3}\rd{x_5x_6}$ belongs to $\ShufPoset(6,3)$, but the words $\bl{y_1y_1}\rd{x_2}\bl{y_3}\rd{x_5x_6}$ and $\bl{y_1y_2}\rd{x_2}\bl{y_3}\rd{x_6x_5}$ do not.

\begin{definition}
    The \defn{indel} operation is the operation on a word $\u \in \ShufPoset(m,n)$ that either deletes a letter from $X$ or inserts a letter from $Y.$
\end{definition}

This operation gives the set of shuffle words a lattice structure. We call this lattice the \defn{shuffle lattice} $\ShufPoset(m,n),$ which has order relation the reflexive and transitive closure of the indel operation, denoted as $\shufleq$~\cite{greene1988posets}. 

\begin{figure}
        \centering
        \includegraphics[scale=1]{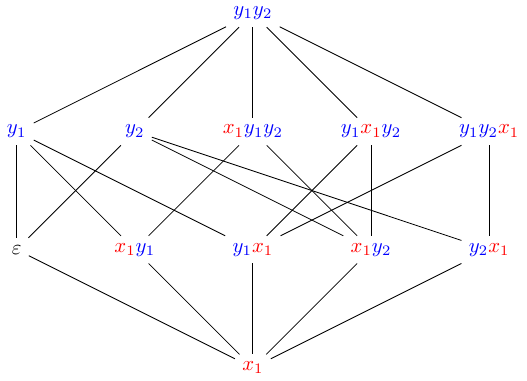}
		\caption{The shuffle lattice $\ShufPoset(1,2)$.}
		\label{fig:shuffle_12}
\end{figure}

If $\u$ and $\v$ are shuffle words, let $\u_{\v}$ be the subword of $\u$ consisting of letters that appear in $\v.$ Additionally, let $\x$ and $\y$ denote the words $x_1\dots x_m$ and $y_1 \dots y_n,$ respectively.  

\section{The $M$-triangle of $\ShufPoset(m,n)$}

We now look at the $M$-triangle, which is a bivariate specialization of the characteristic polynomial of the shuffle lattice $\ShufPoset(m,n).$ 

For a finite poset $\Poset=(P,\leq)$, the \defn{M{\"o}bius function} $\mu_{\Poset}\colon P\times P\to\mathbb{Z}$ is defined recursively by
\begin{displaymath}
    \mu_{\Poset}(u,v) \defs \begin{cases}1, & \text{if}\;u=v,\\-\sum_{u\leq r<v}\mu_{\Poset}(u,r), & \text{if}\;u<v,\\0, & \text{otherwise}.\end{cases}
\end{displaymath}
Additionally, for an element of a graded poset $v \in \Poset,$ let $\rk(v)$ denote the rank function. 

In 2022, McConville and M\"uhle defined an analog of the characteristic polynomial of a graded poset using the M\"obius function. 

\begin{definition}[\cite{bubblelattices2}]
    The \defn{(reverse) characteristic polynomial} of a graded poset $\Poset$ is
\begin{displaymath}
	\ch_{\Poset}(q) \defs \sum_{v\in P}\mu_{\Poset}(\hat{0},v)q^{\rk(v)},
\end{displaymath}
where $\hat{0}$ is the minimum element of $\Poset.$ 
\end{definition}

The $M$-triangle is defined by McConville and M\"uhle as the bivariate specialization of the reverse characteristic polynomial, providing us more combinatorial information than its single-variable analog:
\begin{definition}[{\cite{bubblelattices2}}]
The \defn{$M$-triangle} of a graded poset $\Poset$ is 
\begin{displaymath}
	M_{\Poset}(q,t) \defs \sum_{u\leq v}\mu_{\Poset}(u,v)q^{\rk(u)}t^{\rk(v)}.
\end{displaymath}
\end{definition}

In this paper, we are interested in the $M$-triangle of $\ShufPoset(m,n)$. So, let $\ch_{m,n}(q)$ denote $\ch_{\ShufPoset(m,n)}(q)$, let $\mu_{m,n}(\u,\v)$ denote $\mu_{\ShufPoset(m,n)}(\u,\v),$ and let $M_{m,n}(q,t)$ denote $M_{\ShufPoset(m,n)}(q,t)$. 

The closed form of the reverse characteristic polynomial of a shuffle lattice $\ShufPoset(m, n)$ looks like the following
\begin{prop}[{\cite[Proposition~5.3]{bubblelattices2}}]\label{prop:shuffle_char} 
For $m,n\geq 0$, we have 
	\begin{displaymath}
		\ch_{m,n}(q) = \sum_{a\geq 0}\binom{m}{a}\binom{n}{a}(-q)^{a}(1-q)^{m+n-a}.
	\end{displaymath}
\end{prop}
McConville and M\"uhle also conjecture an analogous explicit form for the $M$-triangle of a shuffle lattice, which is our main result that we repeat here for the convenience of the reader.

\mexplicit*

To prove this theorem, we write the $M$-triangle of a shuffle lattice $\ShufPoset(m, n)$ in terms of the reverse characteristic polynomial, by applying the following lemma of McConville and M\"uhle: 
\begin{lemma}[{\cite[Lemma~5.1]{bubblelattices2}}]\label{lem: mtriangle_char_poly}
    We have 
    \begin{align*}
    M_{\Poset}(q,t) = \sum_{u\in P}(qt)^{\rk(u)}\ch_{[u,\hat{1}]}(t).
\end{align*}
\end{lemma}
Then, we have that 
\begin{align*}
    M_{m,n}(q,t) &= \sum_{\u\in \ShufPoset(m, n)}(qt)^{\rk(\u)}\ch_{[\u,\hat{1}_{\ShufPoset(m, n)}]}(t).
\end{align*}
In $\ShufPoset(m, n),$ the maximal element $\hat{1}_{\ShufPoset(m, n)}$ is $\y.$ Additionally, it is easy to see that $\rk(\u) = |\u_{\y}|+ m - |\u_{\x}|.$ Recall that $|\u_{\y}|$ and $|\u_{\x}|$ are the number of $y_i$'s and $x_i$'s in $\u,$ respectively. Thus, we have that 
\begin{align*}
    M_{m,n}(q,t) &= \sum_{\u\in \ShufPoset(m, n)}(qt)^{|\u_{\y}|+ m - |\u_{\x}|}\ch_{[\u, \y]}(t).
\end{align*}

Fix a shuffle word $\u \in \ShufPoset(m, n).$ The subword $\u_{\y}$ can be written as 
\[\u_{\y} = y_{i_1}\dots y_{i_{k^{\u}}},\]
where $k^{\u} = |\u_{\y}|$ and $1 \leq i_1 < \dots < i_{k^{\u}} \leq n.$

This subword $y_{i_1}\dots y_{i_{k^{\u}}}$ partitions the remaining letters in $\u$ and $\y$ into ${k^{\u}}+1$ parts, possibly empty. Let the tuple $\lambda^{\u} = (\lambda_1^{\u}, \dots, \lambda_{{k^{\u}}+1}^{\u})$ count the size of each of these parts in $\y.$ That is, in $\y,$ there are $\lambda_1^{\u}$ letters to the left of $y_{i_1},$ there are $\lambda_{{k^{\u}}+1}^{\u}$ letters to the right of $y_{i_{k^{\u}}},$ and there are $\lambda_i^{\u}$ letters strictly between $y_{i-1}$ and $y_{i}$ for $2 \leq i \leq {k^{\u}}.$ We define $\eta^{\u} = (\eta_1^{\u}, \dots, \eta_{{k^{\u}}+1}^{\u})$ similarly, as the sizes of the subwords that $y_{i_1}\dots y_{i_{k^{\u}}} = \u_{\y}$ divides $\u_{\x}$ into.

\begin{figure}
	\centering
		\includegraphics[page=1,scale=1]{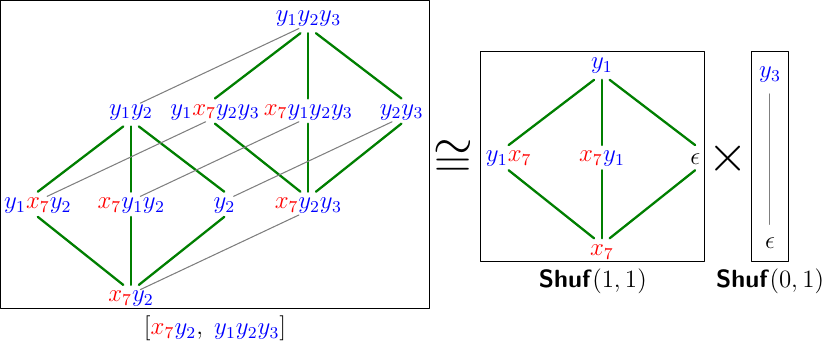}
		\caption{The product decomposition of $[\u, \y]$ in $\ShufPoset(7,3),$ where $\u =\rd{x_7}\bl{y_2}$ and $\y = \bl{y_1y_2y_3}.$ }
		\label{fig:shuf-product}
\end{figure}

\begin{lemma}\label{lem:subposet_decomp}
For all $\u \in \ShufPoset(m,n)$ and $\lambda^{\u}, \eta^{\u}$ defined as above, we have that
\begin{equation*}
    [\u, \y] \cong \ShufPoset(\eta_1^{\u}, \lambda_1^{\u}) \times \dots \times \ShufPoset(\eta_{{k^{\u}}+1}^{\u}, \lambda_{{k^{\u}}+1}^{\u}). 
\end{equation*}
\end{lemma}

\begin{proof}
Let $\u_{\y} = y_{i_1}\dots y_{i_{k^{\u}}}.$ Observe that every element of $[\u, \y]$ can be written in the form 
\begin{displaymath}
    \w_1'y_{i_1}\w_2'\dots \w_{k^{\u}}'y_{i_{k^{\u}}}\w_{{k^{\u}}+1}'
\end{displaymath}
in a unique way, where each word consists of letters $x_{a_s'}$'s such that $\sum_{t = 1}^{j-1}\eta^{\u}_t < a_s' \leq \sum_{t = 1}^{j}\eta^{\u}_t$ and letters $y_{b'}$'s such that $i_{j-1} < b' < i_{j}.$ For each $x_{a_s'}$ in $\w_j',$ let $a_s = a_s' - \sum_{t = 1}^{j-1} \eta^{\u}_t.$ For each $y_{b'},$ let $b = b' - i_{j-1}.$ Let $\w_j$ be the word consisting of $x_{a_s}$'s and $y_b$'s in the same relative order as our $x_{a_s'}$'s and $y_{b'}$'s. This is the word that results from shifting the indices of the $x_{a_s'}$'s in $\w_j'$ back by $\sum_{t = 1}^{j-1} \eta^{\u}_t$ and the $y_{b'}$'s back by $i_{j-1}.$ 

Now, notice that the map from $[\u, \y]$ to $\ShufPoset(\eta_1, \lambda_1) \times \dots \times  \ShufPoset(\eta_{{k^{\u}}+1}, \lambda_{{k^{\u}}+1})$ defined by sending 
\begin{displaymath}
    \w_1'y_{i_1}\w_2'\dots \w_{k^{\u}}'y_{i_{k^{\u}}}\w_{{k^{\u}}+1}' \mapsto (\w_1, \w_2, \dots, \w_{{k^{\u}}+1})
\end{displaymath}
is a bijection. That this bijection is order-preserving is a routine verification. \end{proof}

Additionally, we can show that the reverse characteristic polynomial of a product of posets decomposes nicely: 

\begin{prop}\label{lem:char_poly_product}
Let $\Poset$ and $\Qoset$ be finite, graded posets. Then, 
\begin{displaymath}
    \ch_{\Poset \times \Qoset}(q) = \ch_{\Poset}(q)\ch_{\Qoset}(q).
\end{displaymath}
\end{prop}
\begin{proof}

From Proposition~3.8.2 in~\cite{stanley}, we know that for any two elements $(s, t), (s', t') \in \Poset \times \Qoset,$ we have that $\mu_{\Poset \times \Qoset}((s, t), (s', t')) = \mu_{\Poset}(s, s')\mu_{\Qoset}(t, t').$ It is also easy to see that $\rk_{\Poset \times \Qoset}((s, t)) = \rk_{\Poset}(s) + \rk_{\Qoset}(t)$ for all $(s, t) \in \Poset \times \Qoset.$ Then, using the definition of the reverse characteristic polynomial, we have
\begin{align*}
    \ch_{\Poset \times \Qoset}(q) &= \sum_{(s, t) \in P \times Q} \mu_{\Poset \times \Qoset}(\hat{0}_{\Poset \times \Qoset}, (s, t))q^{\rk_{\Poset \times \Qoset}(s, t)} \\
    &= \sum_{(s, t) \in P \times Q} \mu_{\Poset}(\hat{0}_{\Poset}, s) \mu_{\Qoset}(\hat{0}_{\Qoset}, t) q^{\rk_{\Poset}(s) + \rk_{\Qoset}(t)} \\
    &= \left(\sum_{s \in P} \mu_{\Poset}(\hat{0}_{\Poset}, s) q^{\rk_{\Poset}(s)} \right)\left(\sum_{t \in Q} \mu_{\Qoset}(\hat{0}_{\Qoset}, t) q^{\rk_{\Qoset}(t)}\right) \\
    &= \ch_{\Poset}(q)\ch_{\Qoset}(q), 
\end{align*}
as desired.
\end{proof}

Using Proposition~\ref{lem:char_poly_product}, Lemma~\ref{lem:subposet_decomp} and Proposition~\ref{prop:shuffle_char}, we can then compute an explicit form of the reverse characteristic polynomial of the subposet $[\u, \y]$ of a shuffle lattice $\ShufPoset(m, n):$ 
\begin{align*}
    \ch_{[\u, \y]}(q)  &= \ch_{\ShufPoset(\eta_1^{\u}, \lambda_1^{\u}) \times \dots \times \ShufPoset(\eta_{{k^{\u}}+1}^{\u}, \lambda_{{k^{\u}}+1}^{\u})}(q) \\
    &= \prod_{i = 1}^{{k^{\u}}+1} \ch_{\ShufPoset(\eta_i^{\u}, \lambda_i^{\u})}(q).
\end{align*}
Thus, using the explicit formula of the reverse characteristic polynomial as found in Proposition~\ref{prop:shuffle_char}, we have that 
\begin{equation}\label{eq:subposet_char}
    \ch_{[\u, \y]}(q) = \prod_{i = 1}^{{k^{\u}}+1} \sum_{a \geq 0} {\eta_i^{\u} \choose a}{\lambda_i^{\u} \choose a} (-q)^a (1-q)^a.
\end{equation}

We can use this explicit formula in Equation~\ref{eq:subposet_char} to compute the $M$-triangle using Lemma~\ref{lem: mtriangle_char_poly}: 
\begin{align*}
    M_{m,n}(q,t) &= \sum_{\u\in \ShufPoset(m, n)}(qt)^{|\u_{\y}|+ m - |\u_{\x}|}\ch_{[\u, \y]}(t) \\
    &= \sum_{\u\in \ShufPoset(m, n)} (qt)^{|\u_{\y}|+ m - |\u_{\x}|} \prod_{i = 1}^{{k^{\u}}+1} \sum_{a \geq 0} {\eta_i^{\u} \choose a}{\lambda_i^{\u} \choose a} (-t)^a (1-t)^a.
\end{align*}

Let us now group this sum by the sequences $\lambda^{\u} = (\lambda_1^{\u}, \dots, \lambda_{k^{\u}+1}^{\u})$ and $\eta^{\u} = (\eta_1^{\u}, \dots, \eta_{k^{\u}+1}^{\u}).$ 

Fix an integer $k$ and sequences $\lambda = (\lambda_1, \dots, \lambda_{k+1})$ and $\eta = (\eta_1, \dots, \eta_{k+1})$ of nonnegative integers such that $\lambda_1 + \dots + \lambda_{k+1} = n-k$ and $\eta_1 + \dots + \eta_{k+1} \leq m.$ We will now determine the number of $\u \in \ShufPoset(m,n)$ such that $\lambda^{\u} = \lambda$ and $\eta^{\u} = \eta.$ It is ${m\choose j},$ where $j = \eta_1 + \dots + \eta_{k+1}.$ This is because there is exactly one way to choose $k$ letters from $\y$ to form $\u_{\y}$ so that the remaining $n-k$ letters in $\y$ form consecutive subwords of length $\lambda_1, \dots, \lambda_{k+1}.$ However, we can choose any $j$ of the $m$ letters of $\x$ to form $\u_{\x},$ which we know will be partitioned into consecutive subwords of size $\eta_1, \dots, \eta_{k+1}$ in $\u.$ 

Therefore, we can rewrite our expression for $M_{m,n}(q,t)$ to sum over $k = |\u_{\y}|$ and $j = |\u_{\x}|,$ as well as the sequences $\lambda$ and $\mu.$ 
\allowdisplaybreaks{
\begin{align*}
    M_{m,n}(q,t) &= \sum_{\u\in \ShufPoset(m, n)}(qt)^{|\u_{\y}|+ m - |\u_{\x}|}\ch_{[\u, \y]}(t) \\
    &= \sum_{\u\in \ShufPoset(m, n)} (qt)^{|\u_{\y}|+ m - |\u_{\x}|} \prod_{i = 1}^{k+1} \sum_{a \geq 0} {\eta_i^{\u} \choose a}{\lambda_i^{\u} \choose a} (-q)^a (1-q)^a \\
    &= \sum_{j, k \geq 0} (qt)^{k + m - j} {m \choose j} \sum_{\substack{\eta_1 + \dots + \eta_{k+1} = j\\ \lambda_1 + \dots + \lambda_{k+1}= n-k}} \prod_{i = 1}^{k+1} \sum_{a \geq 0} {\lambda_i \choose a} {\eta_i \choose a} (-t)^a (1-t)^{\eta_i + \lambda_i - a}.
\end{align*}}

To simplify this expression, we set $q$ and $t$ to $-q$ and $-t$ respectively. This yields 
\begin{align*}
    M_{m,n}(-q,-t) &= \sum_{j, k \geq 0} (qt)^{k + m - j} {m \choose j} \sum_{\substack{\eta_1 + \dots + \eta_{k+1} = j\\ \lambda_1 + \dots + \lambda_{k+1}= n-k}} \prod_{i = 1}^{k+1} \sum_{a \geq 0} {\lambda_i \choose a} {\eta_i \choose a} t^a (t+1)^{\eta_i + \lambda_i - a} \\
    &= \sum_{j, k \geq 0} (qt)^{j+k} {m \choose j} \sum_{\substack{\eta_1 + \dots + \eta_{k+1} = m-j\\ \lambda_1 + \dots + \lambda_{k+1}= n-k}} \prod_{i = 1}^{k+1} \sum_{a \geq 0} {\lambda_i \choose a} {\eta_i \choose a} t^a (t+1)^{\eta_i + \lambda_i - a}.
\end{align*}
To show the explicit form of the $M$-triangle as stated in Theorem~\ref{thm:m-explicit-formula}, we show that $\sum_{m, n \geq 0} M_{m, n}(-q, -t) x^m y^n$ has the following rational function expansion: 
\allowdisplaybreaks{
\begin{restatable}{lemma}{lhs}\label{lem:lhs_rational_function}
    We have that
    \begin{align*}
        \sum_{m,n \geq 0} M_{m, n}(-q, -t) x^m y^n &= \frac{1}{(1-x(qt + t+1))(1-y(qt + t+ 1)) - t(t+1)(q+1)xy}.
    \end{align*}
\end{restatable}
\begin{proof}
    We know that 
    \begin{align*}
        &\phantom=\phantom. \sum_{m, n\geq 0} M_{m, n}(-q, -t) x^m y^n \\
        &= \sum_{m, n \geq 0} x^my^n \sum_{j, k \geq 0} (qt)^{j+k} {m \choose j} \sum_{\substack{\eta_1 + \dots + \eta_{k+1} = m-j\\ \lambda_1 + \dots + \lambda_{k+1}= n-k}} \prod_{i = 1}^{k+1} \sum_{a \geq 0} {\lambda_i \choose a} {\eta_i \choose a} t^a (t+1)^{\eta_i + \lambda_i - a}. 
    \end{align*}
    Then, re-indexing so there are no dependencies between the indices $j, k, m$ and $n,$ this expression is equal to 
    \begin{align*}
        &\phantom= \sum_{j, k, m, n \geq 0} x^{m+j} y^{n+k} (qt)^{j+k} {m+j \choose j} \sum_{\substack{\eta_1 + \dots + \eta_{k+1} = m\\ \lambda_1 + \dots + \lambda_{k+1}= n}} \prod_{i = 1}^{k+1} \sum_{a \geq 0} {\lambda_i \choose a} {\eta_i \choose a} t^a (t+1)^{\eta_i + \lambda_i - a}.
    \end{align*}
    In this series, there is only one binomial coefficient that involves $j.$ So, we can apply the power series expansion $\frac{1}{(1-x)^{n+1}} = \sum_{k \geq 0} {{n+k}\choose n}x^k$ to eliminate the sum involving $j:$ 
    \begin{align*}
        &\phantom= \sum_{j, k, m, n \geq 0} x^{m+j} y^{n+k} (qt)^{j+k} {m+j \choose j} \sum_{\substack{\eta_1 + \dots + \eta_{k+1} = m\\ \lambda_1 + \dots + \lambda_{k+1}= n}} \prod_{i = 1}^{k+1} \sum_{a \geq 0} {\lambda_i \choose a} {\eta_i \choose a} t^a (t+1)^{\eta_i + \lambda_i - a} \\
        &= \sum_{k, m, n \geq 0} x^m y^n (qty)^k \left(\sum_{\substack{\eta_1 + \dots + \eta_{k+1} = m\\ \lambda_1 + \dots + \lambda_{k+1}= n}} \prod_{i = 1}^{k+1} \sum_{a \geq 0} {\lambda_i \choose a} {\eta_i \choose a} t^a (t+1)^{\eta_i + \lambda_i - a} \right) \sum_{j \geq 0} {m+j \choose j} (qtx)^j \\
        &= \sum_{k, m, n \geq 0} x^m y^n (qty)^k \left(\sum_{\substack{\eta_1 + \dots + \eta_{k+1} = m\\ \lambda_1 + \dots + \lambda_{k+1}= n}} \prod_{i = 1}^{k+1} \sum_{a \geq 0} {\lambda_i \choose a} {\eta_i \choose a} t^a (t+1)^{\eta_i + \lambda_i - a}\right) \frac{1}{(1-qtx)^{m+1}} \\
        &= \frac{1}{1- qtx} \sum_{k, m, n \geq 0} \left(\frac{x}{1-qtx}\right)^m y^n (qty)^k \sum_{\substack{\eta_1 + \dots + \eta_{k+1} = m\\ \lambda_1 + \dots + \lambda_{k+1}= n}} \prod_{i = 1}^{k+1} \sum_{a \geq 0} {\lambda_i \choose a} {\eta_i \choose a} t^a (t+1)^{\eta_i + \lambda_i - a}. 
    \end{align*}
    That is, 
    \begin{equation}\label{eq:M-to-f}
        \sum_{m, n \geq 0} M_{m, n} (-q, -t) x^m y^n = \frac{1}{1- qtx} f(q, t, x, y),
    \end{equation}
    where $f(q, t, x, y)$ is defined to be 
    \begin{align*}
        f(q, t, x, y) = \sum_{k, m, n \geq 0} \left(\frac{x}{1-qtx}\right)^m y^n (qty)^k \sum_{\substack{\eta_1 + \dots + \eta_{k+1} = m\\ \lambda_1 + \dots + \lambda_{k+1}= n}} \prod_{i = 1}^{k+1} \sum_{a \geq 0} {\lambda_i \choose a} {\eta_i \choose a} t^a (t+1)^{\eta_i + \lambda_i - a}.
    \end{align*}
    
    We will now compute $f(q, t, x, y).$ We remove the dependencies between $m$ and $\eta_1, \dots, \eta_{k+1}$ and $n$ and $\lambda_1, \dots, \lambda_{k+1}$ by summing over all $\eta_i, \lambda_i \geq 0$ and setting $m = \eta_1 + \dots + \eta_{k+1}$ and $n = \lambda_1 + \dots + \lambda_{k+1},$ which gives us  
    \begin{align*}
        &\phantom= f(q, t, x, y) \\
        &= \sum_{k, m, n \geq 0} \left(\frac{x}{1-qtx}\right)^m y^n (qty)^k \sum_{\substack{\eta_1 + \dots + \eta_{k+1} = m\\ \lambda_1 + \dots + \lambda_{k+1}= n}} \prod_{i = 1}^{k+1} \sum_{a \geq 0} {\lambda_i \choose a} {\eta_i \choose a} t^a (t+1)^{\eta_i + \lambda_i - a} \\
        &= \sum_{k \geq 0} (qty)^k \sum_{\substack{\eta_1, \dots, \eta_{k+1}, \\ \lambda_1, \dots, \lambda_{k+1} \geq 0}} \left(\frac{x}{1-qtx}\right)^{\eta_1 + \dots +\eta_{k+1}} y^{\lambda_1 + \dots + \lambda_{k+1}} \prod_{i = 1}^{k+1} \sum_{a \geq 0} {\lambda_i \choose a} {\eta_i \choose a} t^a (t+1)^{\eta_i + \lambda_i - a} \\
        &= \sum_{k \geq 0} (qty)^k \prod_{i = 1}^{k+1} \sum_{\eta_i, \lambda_i \geq 0} \left(\frac{x}{1-qtx}\right)^{\eta_i} y^{\lambda_i} \sum_{a \geq 0} {\lambda_i \choose a} {\eta_i \choose a} t^a (t+1)^{\eta_i + \lambda_i - a}.
    \end{align*}
    Now, we re-index to remove the dependencies between $a$ and $\eta_i, \lambda_i:$ 
    \begin{align*}
        &\phantom= \sum_{k \geq 0} (qty)^k \prod_{i= 1}^{k+1}\sum_{\eta_i, \lambda_i \geq 0} \left(\frac{x}{1-qtx}\right)^{\eta_i} y^{\lambda_i} \sum_{a \geq 0} {\lambda_i \choose a} {\eta_i \choose a} t^a (t+1)^{\eta_i + \lambda_i - a} \\
        &= \sum_{k \geq 0} (qty)^k \prod_{i = 1}^{k+1} \sum_{\eta_i, \lambda_i, a \geq 0} \left(\frac{x}{1-qtx}\right)^{\eta_i + a} y^{\lambda_i + a} {\lambda_i + a \choose a} {\eta_i + a\choose a} t^a (t+1)^{\eta_i + \lambda_i + a}.
    \end{align*}
    Again, we notice that there is only one binomial coefficient involving $\eta_i$ and one binomial coefficient involving $\lambda_i,$ so we can eliminate the sums involving these variables  
    \begin{align*}
        &\phantom= \sum_{k \geq 0} (qty)^k \prod_{i = 1}^{k+1} \sum_{\eta_i, \lambda_i, a \geq 0} \left(\frac{x(t+1)}{1-qtx}\right)^{\eta_i} (y(t+1))^{\lambda_i + a} \left(\frac{xyt(t+1)}{1-qtx}\right)^a {\lambda_i + a \choose a} {\eta_i + a\choose a}\\
        &= \sum_{k \geq 0} (qty)^k \prod_{i = 1}^{k+1} \sum_{a \geq 0} \left(\frac{xyt(t+1)}{1-qtx}\right)^a \frac{1}{\left(1 - \frac{x(t+1)}{1-qtx}\right)^{a+1} \left(1 - y(t+1)\right)^{a+1}}. 
    \end{align*}
    Next, we eliminate the sum involving $a$ by using the closed form of an infinite geometric series.
    \begin{align*}
        &\phantom= \sum_{k \geq 0} (qty)^k \prod_{i = 1}^{k+1} \frac{1 - qtx}{\left(1- qtx - x(t+1)\right)(1 - y(t+1))}\sum_{a \geq 0} \left(\frac{xyt(t+1)}{(1 - qtx - x(t+1))\left(1 - y(t+1)\right)}\right)^a \\
        &= \sum_{k \geq 0} (qty)^k \prod_{i = 1}^{k+1} \left(\frac{1 - qtx}{\left(1- qtx - x(t+1)\right)(1 - y(t+1))}\right)\left( \frac{1}{1- \frac{xyt(t+1)}{(1 - qtx - x(t+1))\left(1 - y(t+1)\right)}}\right) \\
        &= \sum_{k \geq 0} (qty)^k \left(\frac{1 - qtx}{\left(1- qtx - x(t+1)\right)(1 - y(t+1)) - xyt(t+1)} \right)^{k+1}. 
    \end{align*}
    That is, 
    \begin{equation}\label{eq:f-to-g}
        f(q, t, x, y) = \frac{1-qtx}{\left(1- qtx - x(t+1)\right)(1 - y(t+1)) - xyt(t+1)} g(q, t, x, y),
    \end{equation}
    where $g(q, t, x, y)$ is defined to be 
    \begin{align*}
        g(q, t, x, y) &= \sum_{k \geq 0} \left(\frac{qty(1 - qtx)}{\left(1- qtx - x(t+1)\right)(1 - y(t+1)) - xyt(t+1)} \right)^{k}.
    \end{align*}
    Then, we can re-write the expression for $\sum_{m,n \geq 0} M_{m, n}(-q, -t) x^m y^n$ to be in terms of $g(q, t, x, y)$ using Equation~\ref{eq:M-to-f} and Equation~\ref{eq:f-to-g}:
    \begin{align*}
        \sum_{m,n \geq 0} M_{m, n}(-q, -t) x^m y^n &= \frac{1}{1-qtx} f(q, t, x, y) \\
        &= \frac{1}{1-qtx} \frac{1-qtx}{\left(1- qtx - x(t+1)\right)(1 - y(t+1)) - xyt(t+1)} g(q, t, x, y) \\
        &= \frac{1}{\left(1- qtx - x(t+1)\right)(1 - y(t+1)) - xyt(t+1)}g(q, t, x, y).
    \end{align*}
    Finally, we determine $g(q, t, x, y)$ using the closed form of a geometric series, which yields the rational function 
    \begin{align*}
        g(q, t, x, y) &= \sum_{k \geq 0} \left(\frac{qty(1 - qtx)}{\left(1- qtx - x(t+1)\right)(1 - y(t+1)) - xyt(t+1)} \right)^{k} \\
        &= \frac{1}{1 - \frac{qty(1 - qtx)}{\left(1- qtx - x(t+1)\right)(1 - y(t+1)) - xyt(t+1)}} \\
        &=  \frac{\left(1- qtx - x(t+1)\right)(1 - y(t+1)) - xyt(t+1)}{\left(1- qtx - x(t+1)\right)(1 - y(t+1)) - xyt(t+1) - qty(1 - qtx)}.
    \end{align*}
    Therefore, we have 
    \begin{align*}
        \sum_{m,n \geq 0} M_{m, n}(-q, -t) x^m y^n &= \frac{1}{\left(1- qtx - x(t+1)\right)(1 - y(t+1)) - xyt(t+1)}g(q, t, x, y) \\
        &= \frac{1}{\left(1- qtx - x(t+1)\right)(1 - y(t+1)) - xyt(t+1) - qty(1 - qtx)} \\
        &= \frac{1}{(1-x(qt + t+1))(1-y(qt + t+ 1)) - t(t+1)(q+1)xy},
    \end{align*}
    as desired. 
\end{proof}

}

Then, using Lemma~\ref{lem:lhs_rational_function}, we can find the rational function representation of the $M$-triangle, by setting $-q$ to $q$ and $-t$ to $t,$ which gives us 
\begin{equation}\label{eq:m-triangle-rational}
\sum_{m, n \geq 0} M_{m, n}(q, t) x^m y^n = \frac{1}{(1-x(qt - t+1))(1-y(qt - t+ 1)) + t(1-t)(q+1)xy},
\end{equation}
proving Theorem~\ref{thm:m-rational-function} as introduced at the beginning of the paper.

Alternatively, we can prove Lemma~\ref{lem:lhs_rational_function} by first proving the following identity using a combinatorial argument, which we include in the appendix of the paper. 
\begin{restatable}{lemma}{theidentity}\label{lem:inner_sum}
    For $m, n, k \geq 0,$ we have
    \begin{equation*}
        \sum_{\substack{\eta_1 + \dots + \eta_{k+1} = m\\ \lambda_1+ \dots + \lambda_{k+1} = n}} \prod_{i = 1}^{k+1} \sum_{a \geq 0} {\lambda_i \choose a} {\eta_i \choose a} t^a (t+1)^{\eta_i - a} = {n+ k \choose k} \sum_{\ell \geq 0} {{m+k}\choose {\ell + k}}{{n+ k + \ell} \choose \ell} t^{\ell},
    \end{equation*}
    where the left hand side is a sum over all sequences of nonnegative integers $(\eta_1, \dots, \eta_{k+1})$ and $(\lambda_1, \dots, \lambda_{k+1})$ such that $\eta_1 + \dots + \eta_{k+1} = m$ and $\lambda_1+ \dots + \lambda_{k+1} = n.$ 
\end{restatable}

By performing a similar process as in the proof of Lemma~\ref{lem:lhs_rational_function}, we can show that the explicit formula for the $M$-triangle has the same rational function representation: 
\begin{lemma}\label{lem:rhs_rational_function}
    We have
    \begin{align*}
        &\phantom= \sum_{m, n \geq 0}x^m y^n \sum_{a \geq 0} {n \choose a}{m \choose a} t^a (t+1)^a (q+1)^a (qt + t + 1)^{m+n-2a} \\
        &= \frac{1}{(1-x(qt + t+1))(1-y(qt + t+ 1)) - t(t+1)(q+1)xy}.
    \end{align*}
\end{lemma}
\begin{proof}
    First, we change the order of summations, giving us 
    \begin{align*}
        &\phantom= \sum_{m,n \geq 0} x^m y^n \sum_{a \geq 0} {n \choose a}{m \choose a} t^a (t+1)^a (q+1)^a (qt + t + 1)^{m+n-2a} \\
        &= \sum_{a \geq 0} \left(\frac{t(t+1)(q+1)}{(qt+t+1)^2}\right)^a \sum_{m \geq a} {m \choose a} (x(qt+t+1))^m \sum_{n \geq a} {n \choose a} (y(qt+t+1))^n.
    \end{align*}
    Then, by applying the power series expansion $\frac{1}{(1-x)^{n+1}} = \sum_{k \geq 0} {{n+k}\choose n}x^k$ twice, we have that this sum is equal to 
    \begin{align*}
        &\phantom= \sum_{a \geq 0} \left(\frac{t(t+1)(q+1)}{(qt+t+1)^2}\right)^a \sum_{m \geq a} {m \choose a} (x(qt+t+1))^m \sum_{n \geq a} {n \choose a} (y(qt+t+1))^n \\
        &= \sum_{a \geq 0} \left(\frac{t(t+1)(q+1)}{(qt+t+1)^2}\right)^a \left(\frac{(x(qt+t+1))^a}{(1- x(qt+t+1))^{a+1}}\right)\left( \frac{(y(qt+t+1))^a}{(1- y(qt+t+1))^{a+1}}\right).
    \end{align*}
    This is an infinite geometric series with rational function 
    \begin{align*}
        \frac{\frac{1}{(1- x(qt+t+1))(1- y(qt+t+1))}}{1 - \frac{t(t+1)(q+1)xy}{(1- x(qt+t+1))(1- y(qt+t+1))}} 
        &= \frac{1}{(1-x(qt + t+1))(1-y(qt + t+ 1)) - t(t+1)(q+1)xy},
    \end{align*}
    so we are done. 
\end{proof}

Since the rational function representation of the explicit formula found in Lemma~\ref{lem:rhs_rational_function} matches with the rational function representation of the $M$-triangle in Equation~\ref{eq:m-triangle-rational}, we can conclude that the $M$-triangle has the following explicit formula, which resolves \cite[Conjecture~5.6]{bubblelattices2}. We restate this result here for the reader's convenience:
\mexplicit*

\section{Relating the $M$-triangle and the $H$-triangle}

Finally, we offer a surprising connection between the $M$-triangle and the $H$-triangle, which is a bivariate variation of the rank generating function of the shuffle lattice.

To define the $H$-triangle, we must consider another operation on shuffle words called a \defn{(forward) transposition} $\u \rightarrow \v,$ where the word $\v$ is obtained from $\u$ by reversing an adjacent pair $x_i y_j.$ In~\cite{bubblelattices1}, McConville and M\"uhle define a \defn{bubble lattice} $\Bub(m,n)$ by placing a new partial order on $\ShufPoset(m,n)$ defined by both indels and transpositions. The order relation of $\Bub(m,n)$ is the reflexive and transitive closure of the indel and transposition operation, denoted by $\bubleq.$ 

\begin{figure}
    \centering
        \includegraphics[scale=1]{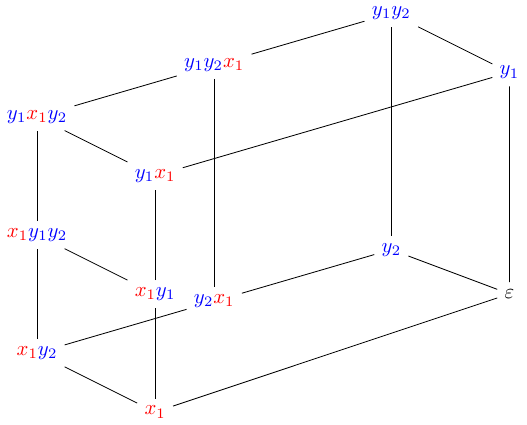}
		\caption{The bubble lattice $\Bub(1,2)$.}
		\label{fig:bubble_12}
\end{figure}

McConville and M\"uhle also provide a characterization of the covering pairs of the bubble lattice in terms of transpositions and a particular type of indel. Let $\u_{\hat{\i}}$ denote the word obtained by deleting letter $u_i$ from $\u.$ Specifically, a \defn{right indel} is defined by
\begin{displaymath}
	\v\sindel\v' \quad\text{if and only if}\quad \begin{cases}\v=\u\;\text{and}\;\v'=\u_{\hat{\i}}, & \text{if}\;u_{i}\in X\;\text{and either}\;u_{i+1}\in X\;\text{or}\;i=k,\\ \v=\u_{\hat{\i}}\;\text{and}\;\v'=\u, & \text{if}\;u_{i},u_{i+1}\in Y.\end{cases}.
\end{displaymath}
If $u_{i}\in X$ and $u_{i+1}\in Y$, then a \defn{(forward) transposition} is a transformation $\u\transpose\u'$, where $\u'=u_{1}u_{2}\cdots u_{i-1}u_{i+1}u_{i}u_{i+2}\cdots u_{k}$. 
\begin{lemma}[{\cite[Lemma~3.5]{bubblelattices1}}]\label{lem:bubble_covers}
For $\u, \v \in \ShufPoset(m,n)$ we have $\u \bubcov \v$ if and only if either $\u\transpose\v$ or $\u\sindel\v.$
\end{lemma}
The \defn{in-degree} of $\v \in \ShufPoset(m,n)$ counts the number of shuffle words that $\v$ covers:
\begin{displaymath}
	\indeg(\v) \defs \bigl\lvert\bigl\{\u\in\ShufPoset(m,n)\colon \u\bubcov\v\bigr\}\bigr\rvert.
\end{displaymath}
Lemma~\ref{lem:bubble_covers} tells us that the covering pairs in $\Bub(m,n)$ can be split into two types. To realize this, the authors of~\cite{bubblelattices2} define the \defn{indel-degree} of $\u$ to be
\begin{displaymath}
	\indeld(\u) \defs \bigl\lvert\bigl\{\u'\in\ShufPoset(m,n)\colon \u'\sindel\u\bigr\}\bigr\rvert,
\end{displaymath}
and the \defn{transpose-degree} of $\u$ to be
\begin{displaymath}
  \transd(\u) \defs \bigl\lvert\bigl\{\u'\in\ShufPoset(m,n)\colon \u'\transpose\u\bigr\}\bigr\rvert.
\end{displaymath}
Clearly, we have $\indeg(\u) = \indeld(\u) + \transd(\u)$. 

We now define the $H$-triangle:
\begin{definition}[{\cite{bubblelattices2}}]
For $m, n \geq 0,$ the \defn{$H$-triangle} of $\Bub(m,n)$ is
\begin{equation*}
	H_{m,n}(q,t) \defs \sum_{\u\in\ShufPoset(m,n)}q^{\indeg(\u)}t^{\indeld(\u)}.
\end{equation*}
\end{definition}
In~\cite{bubblelattices2}, the authors note that the polynomial $H_{m,n}(q,1)$ appears already in as Corollary~4.8 in Greene's paper as the rank-generating polynomial of $\ShufPoset(m,n)$, and was used to establish the rank symmetry of $\ShufPoset(m,n)$ and for constructing a decomposition into symmetrically placed Boolean lattices~\cite{greene1988posets}. They also find an explicit formula for $H_{m,n}(q,t)$:
\begin{theorem}[{\cite[Theorem~1.3]{bubblelattices2}}]\label{thm:h_explicit}
For $m,n\geq 0$, we have
	\begin{displaymath}
		H_{m,n}(q,t) = \sum_{a=0}^{\min\{m,n\}} \binom{m}{a}\binom{n}{a}q^a(qt+1)^{m+n-2a}.
	\end{displaymath}
\end{theorem}
From this explicit form and the explicit formula for the characteristic polynomial in Proposition~\ref{prop:shuffle_char}, the authors show the following relationship between the $H$-triangle and $M$-triangle:
\begin{cor}[{\cite[Corollary~5.4]{bubblelattices2}}]\label{cor:char_from_h}
For $m,n\geq 0$, we have
	\begin{displaymath}
		\ch_{m,n}(q) = q^{m+n}H_{m,n}\left(\frac{q-1}{q},\frac{1-2q}{q-1}\right).
	\end{displaymath}
\end{cor}
We can then show an analogous relationship between the $H$-triangle and $M$-triangle, resolving Conjecture~5.5 in \cite{bubblelattices2}:
\fromhtom*
Corollary~\ref{cor:h-to-m} follows from comparing the explicit formulas found in Theorem~\ref{thm:m-explicit-formula} and Theorem~\ref{thm:h_explicit}. However, we still are not aware of a conceptual explanation of this relation, nor of Corollary~\ref{cor:char_from_h}, which the author encourages the readers to find. 

\section*{Acknowledgments}

This research was conducted at the Duluth REU at the University of Minnesota Duluth, which is supported by National Science Foundation grant No. DMS-2409861, Jane Street Capital, and personal donations from Ray Sidney and Eric Wepsic. I would like to thank Evan Chen, Colin Defant, Noah Kravitz, Maya Sankar, and Carl Schildkraut for providing guidance during the research process. I would also like to extend special thanks to Mitchell Lee, Katherine Tung, and Noah Kravitz for their crucial feedback throughout the editing process. Finally, I am very grateful to Joe Gallian and Colin Defant for their support and for extending an invitation to participate in the Duluth REU. 

\bibliographystyle{plain}
\bibliography{M-Triangle}

\pagebreak 

\section*{Appendix}

In this appendix, we provide an alternate proof to Theorem~\ref{thm:m-rational-function} using Lemma~\ref{lem:inner_sum}, which we restate here for convenience:
\theidentity*
\begin{proof}
    To begin, we expand the binomial on the left hand side of the equation and simplify, giving us
    \begin{align*}
        \sum_{a \geq 0} {\lambda_i \choose a} {\eta_i \choose a} t^a (t+1)^{\eta_i - a} &= \sum_{a \geq 0} {\lambda_i \choose a} {\eta_i \choose a} t^a \sum_{b \geq 0} {\eta_i - a \choose b} t^b  \\
        &= \sum_{j \geq 0} \sum_{a \geq 0} {\lambda_i \choose a} {\eta_i \choose \eta_i - a}{\eta_i - a \choose \eta_i - j} t^{j} \\
        &= \sum_{j \geq 0} {\eta_i \choose \eta_i - j} {{\lambda_i + j} \choose j} t^j
    \end{align*}
    via a variable substitution and Vandermonde's identity. So, now we know that 
    \begin{align*}
        \sum_{\substack{\eta_1 + \dots +\eta_{k+1} = m\\ \lambda_1 + \dots + \lambda_{k+1}= n}} \prod_{i = 1}^{k+1} \sum_{a \geq 0} {\lambda_i \choose a} {\eta_i \choose a} t^a (t+1)^{\eta_i - a} = \sum_{\substack{\eta_1 + \dots +\eta_{k+1} = m\\ \lambda_1 + \dots + \lambda_{k+1}= n}} \prod_{i = 1}^{k+1} \sum_{j \geq 0} {{\lambda_i + j} \choose j} {\eta_i \choose \eta_i - j} t^{j}. 
    \end{align*}

Suppose that we have a row of $m$ distinguishable green balls and a row of $n$ distinguishable red balls in a fixed order. We place $k$ dividers to split the row of green balls into sections of size $\eta_1, \ldots, \eta_{k+1},$ and place another $k$ dividers to split the row of red balls into sections of size $\lambda_1, \ldots, \lambda_{k+1}.$ Notice that this placement of dividers induces a pair of compositions $(\eta_1, \ldots, \eta_{k+1})$ and $(\lambda_1, \ldots, \lambda_{k+1})$ of $m$ and $n,$ respectively, into $k+1$ parts. Then, the coefficient of $t^{\ell}$ in the expanded product 
\[\prod_{i = 1}^{k+1} \sum_{j\geq 0} {{\lambda_i + j} \choose j} {\eta_i \choose \eta_i - j} t^{j}\]
is the number of ways to choose $m - \ell$ green balls to put a star sticker on, and $\ell$ balls from the remaining $n+\ell$ balls to put a triangle sticker on, such that within each pair of sections $(\eta_i, \lambda_i),$ the number of red balls with a sticker in $\lambda_i$ is equal to the number of green balls in $\eta_i$ with no sticker. This additional restriction is due to the fact that within each pair of parts $(\eta_i, \lambda_i),$ we put a sticker on a total of $\eta_i$ balls. 

Then, if we sum over all compositions of $m$ and $n$ into $k+1$ parts, 
\begin{equation}\label{eq:3sum} 
        \sum_{\substack{\eta_1 + \dots +\eta_{k+1} = m\\ \lambda_1 + \dots + \lambda_{k+1}= n}} \prod_{i = 1}^{k+1} \sum_{j = 0}^{\eta_i} {{\lambda_i + j} \choose j} {\eta_i \choose \eta_i - j} t^{j},
\end{equation}
the coefficient of $t^{\ell}$ in this sum is the number of ways to 
\begin{enumerate}
    \item place the $k$ dividers in a row of $m$ green balls,
    \item place the $k$ dividers in a row of $n$ red balls,
    \item choose $m-\ell$ green balls to put a star sticker on,
    \item and choose $\ell$ balls to put a triangle sticker on from all the non-stickered balls,
\end{enumerate}
all subject to the restriction that within each pair of parts $(\eta_i, \lambda_i)$ in the compositions $\eta_1 + \dots + \eta_{k+1} = m$ and $\lambda_1 + \dots + \lambda_{k+1} =n$ induced by the dividers, the number of red balls with a sticker is equal to the number of green balls without a sticker. 

Then, suppose that we chose r red balls to put stickers on, for some $r \geq 0.$ Notice that these stickers must all be triangles. Then, $m - r$ of the green balls will have stickers on them, and $\ell - r$ of the green balls will have triangle stickers. For a fixed $r,$ we can rewrite items (1) through (4) in the list above as the number of ways to 
    \begin{enumerate}
        \item place $k$ dividers among $m$ green balls in a line to induce a composition $(\eta_1, \dots, \eta_{k+1})$ of $m$ into $k+1$ parts, 
        \item place $k$ dividers among $n$ red balls in a line to induce a composition $(\lambda_1, \dots, \lambda_{k+1})$ of $n,$
        \item choose $r$ red balls to put triangle stickers on,
        \item choose $m - r$ green balls to put stickers on, 
        \item and choose $m - \ell$ green balls to specifically have star stickers, while the other $\ell - r$ green balls have triangles, 
    \end{enumerate}
    all subject to the restriction that within each pair of parts $(\eta_i, \lambda_i)$ within our compositions, the number of red balls with stickers is equal to the number of green balls without stickers. We find this number and then sum over all $r$ to find the coefficient of $t^{\ell}$ in Equation~\ref{eq:3sum}. 
    
    Notice that the order we do the items above does not matter, as long as all of the restrictions are met. 
    \begin{itemize}
        \item There are ${n+k \choose k}$ ways to place $k$ dividers among our $n$ red balls, splitting them into parts of size $(\lambda_1, \dots, \lambda_{k+1})$ (item 2). 
        \item Then, there are ${n \choose r}$ ways to choose $r$ of these red balls to put a triangle sticker on (item 3). 
        \item We find the number of ways to do items 1 and 4 together. In any valid placement of $k$ dividers among our $m$ green balls inducing the composition $(\eta_1, \dots, \eta_{k+1})$, the number of stickered red balls in each $\lambda_i$ must be equal to the number of non-stickered green balls in the corresponding section $\eta_i.$ This means that the relative order of the $k$ dividers and the $r$ non-stickered green balls is determined by how we already placed the dividers among the red balls. Then, all that remains to do is to determine the positions of the stickered green balls relative to the dividers and non-stickered green balls. There are ${m-r + (r+k) \choose r + k} = {{m+k} \choose r + k}$ ways to arrange the $m-r$ stickered green balls among the $r$ non-stickered green balls and $k$ dividers.
        \item Finally, there are ${m-r \choose m-\ell}$ ways choose $m- \ell$ of the stickered green balls to have a star sticker, and the remaining stickered green balls to have a triangle (item 5). 
    \end{itemize}
    So, in total, for a fixed $r,$ the number of ways to divide and put stickers on our red and green balls is
    \[{n \choose r}{{n + k} \choose k}{{m + k} \choose {r + k}}{{m - r \choose m - \ell}}.\]
    
    If we sum over all $r \geq 0,$ we get that the coefficient of $t^{\ell}$ in Equation~\ref{eq:3sum} is 
    \begin{align*}
        \sum_{r \geq 0}{n \choose r}{{n + k} \choose k}{{m + k} \choose {r + k}}{{m - r \choose m - \ell}}. 
    \end{align*}

    Finally, we show that 
    \[\sum_{r \geq 0} {n \choose r}{{m + k} \choose {r + k}}{{m - r \choose \ell - r}} = {m+k \choose \ell + k}{n+k + \ell \choose \ell}\]
    via another counting argument, again using balls and stickers. Suppose we have a collection of $m+k$ green balls, and a collection of $n$ red balls. Then, for a fixed $a,$ the expression ${n \choose a}{{m + k} \choose {a + k}}{{m - a \choose \ell - a}}$ counts the number of ways to choose $a$ red balls to put a star sticker on, $a+k$ green balls to put a triangle sticker on, and $\ell - a$ green balls from the remaining $m-a$ green balls to put a star sticker on. So, in total, we are putting a sticker on $\ell+k$ green balls, and a star on $\ell$ balls out of the $n$ red balls and $\ell + k$ stickered green balls, where $\ell - a$ of the balls with stars are red. Then, if we sum over all $a,$ we are counting the number of ways to first choose $\ell+k$ green balls to put a triangle sticker on, and then $\ell$ objects to put a star sticker out of the $n$ red balls and the $\ell + k$ green balls with a sticker, which is ${m + k \choose \ell + k}{n+ k + \ell \choose \ell}.$
    Using these simplifications, we have
        \begin{align*}
        \sum_{\substack{\eta_1+ \dots+ \eta_{k+1} = m\\ \lambda_1+ \dots+ \lambda_{k+1} = n}} \prod_{i = 1}^{k+1} \sum_{a \geq 0} {\lambda_i \choose a} {\eta_i \choose a} t^a (t+1)^{\eta_i - a} &= \sum_{\substack{\eta_1+ \dots+ \eta_{k+1} = m\\ \lambda_1+ \dots+ \lambda_{k+1} = n}} \prod_{i = 1}^{k+1} \sum_{j \geq 0} {{\lambda_i + j} \choose j} {\eta_i \choose \eta_i - j} t^{j} \\
        &= \sum_{\ell\geq 0}\sum_{r \geq 0} {n \choose r}{{n + k} \choose k}{{m + k} \choose {r + k}}{{m - r \choose m - \ell}}t^{\ell} \\
        &= {{n + k} \choose k}\sum_{\ell \geq 0}{m+k \choose \ell + k}{n+k + \ell \choose \ell}t^{\ell},
    \end{align*}
    as desired. 
\end{proof}

This identity simplifies our expression for $M_{m, n}(-q, -t),$ which is
\begin{align*}
    M_{m,n}(-q,-t) &= \sum_{j, k \geq 0} (qt)^{j+k} {m \choose j} \sum_{\substack{\eta_1 + \dots + \eta_{k+1} = m-j\\ \lambda_1 + \dots + \lambda_{k+1}= n-k}} \prod_{i = 1}^{k+1} \sum_{a \geq 0} {\lambda_i \choose a} {\eta_i \choose a} t^a (t+1)^{\eta_i + \lambda_i - a} \\
    &= \sum_{j, k \geq 0} (qt)^{j+k} {m \choose j} (t +1)^{n-k} \sum_{\substack{\eta_1 + \dots + \eta_{k+1} = m-j\\ \lambda_1 + \dots + \lambda_{k+1}= n-k}} \prod_{i = 1}^{k+1} \sum_{a \geq 0} {\lambda_i \choose a} {\eta_i \choose a} t^a (t+1)^{\eta_i - a}
\end{align*}
where the second equality lies in the fact that $\prod_{i = 1}^{k+1}(t+1)^{\lambda_i} = (t+1)^{n-k}$ for all compositions $\lambda$ of $n-k.$ 

By applying Lemma~\ref{lem:inner_sum} to our expression for the $M$-triangle, we have that 
\begin{align*}
    M_{m, n}(-q, -t) &= \sum_{j, k \geq 0} {m \choose j} (qt)^{k + j} (t+1)^{n-k} \sum_{\substack{\eta_1+ \dots+ \eta_{k+1} = m-j\\ \lambda_1+ \dots+ \lambda_{k+1} = n-k}} \prod_{i = 1}^{k+1} \sum_{a \geq 0} {\lambda_i \choose a} {\eta_i \choose a} t^a (t+1)^{\eta_i - a} \\
    &= \sum_{j, k \geq 0}{n \choose k} {m \choose j} (qt)^{k+j} (t+1)^{n-k} \sum_{\ell \geq 0}{{m-j+k}\choose {\ell + k}}{{n+ \ell} \choose \ell} t^{\ell}.
\end{align*}

We use this to provide an alternate proof of Lemma~\ref{lem:lhs_rational_function}, which we restate here for convenience.
\lhs*

\begin{proof}
    We know that 
    \begin{align*}
        &\phantom=\phantom. \sum_{m, n\geq 0} M_{m, n}(-q, -t) x^m y^n \\
        &= \sum_{m, n \geq 0} x^m y^n \sum_{j, k \geq 0} {n \choose k} {m \choose j} (-qt)^{k+j} (t+1)^{n-k} \sum_{\ell \geq 0} {{m-j+k}\choose {\ell + k}}{{n+ \ell} \choose \ell} t^{\ell}. 
    \end{align*}
    Then, re-indexing so there are no dependencies between indices, this expression is equal to 
    \begin{align*}
        &\phantom= \sum_{j, k, l, m, n \geq 0} (qtx)^{j}  (qty)^k x^{m + \ell} (y(1+t))^n t^{\ell} {n+k \choose k} {m + \ell +j \choose j} {{m+ \ell +k}\choose {\ell + k}} {{n +k + \ell} \choose \ell},
    \end{align*}
    where the 5 indices are independent and can be freely interchanged. Notice that in this series, there is only one binomial coefficient that involves $j.$ So, we can apply the power series expansion $\frac{1}{(1-x)^{n+1}} = \sum_{k \geq 0} {{n+k}\choose n}x^k$ to eliminate the sum involving $j$:
    \begin{align*}
        &\phantom= \sum_{j, k, \ell, m, n \geq 0} (qtx)^{j}  (qty)^k x^{m + \ell} (y(1+t))^n t^{\ell} {n+k \choose k} {m + \ell +j \choose j} {{m+ \ell +k}\choose {\ell + k}} {{n +k + \ell} \choose \ell} \\
        &= \sum_{k, \ell, m, n \geq 0} (qty)^k x^m (y(1+t))^n (xt)^{\ell} {n+k \choose k} {{m+ \ell +k}\choose {\ell + k}} {{n +k + \ell} \choose \ell} \sum_{j \geq 0} {m + \ell +j \choose j} (qtx)^{j} \\
        &= \sum_{k, \ell, m, n \geq 0} (qty)^k x^m (y(1+t))^n (xt)^{\ell} {n+k \choose k} {{m+ \ell +k}\choose {\ell + k}} {{n +k + \ell} \choose \ell} \frac{1}{(1-qtx)^{m + \ell + 1}} \\
        &= \frac{1}{1-qtx} f(q, t, x, y), 
    \end{align*}
    where $f(q, t, x, y)$ is defined to be
    \[
    \sum_{k, \ell, m, n \geq 0} (qty)^k \left(\frac{xt}{1-qtx}\right)^{\ell} (y(t+1))^n \left(\frac{x}{1-qtx}\right)^m {n+ k \choose k} {n+k + \ell \choose n+k} {{m+ \ell +k}\choose {\ell + k}}. 
    \]
    Then, notice that there is now only one binomial coefficient in our expression for $f$ that involves $m,$ so we can similarly eliminate the sum including $m$: 
    \begin{align*}
        &\phantom= f(q, t, x, y) \\
        &= \sum_{k, \ell, m, n \geq 0} (qty)^k \left(\frac{xt}{1-qtx}\right)^{\ell} (y(t+1))^n \left(\frac{x}{1-qtx}\right)^m {n+ k \choose k} {n+k + \ell \choose n+k} {{m+ \ell +k}\choose {\ell + k}} \\
        &=  \sum_{k, \ell, n \geq 0} (qty)^k \left(\frac{xt}{1-qtx}\right)^{\ell} (y(t+1))^n {n+ k \choose k} {n+k + \ell \choose n+k} \sum_{m \geq 0} {{m+ \ell +k}\choose {\ell + k}} \left(\frac{x}{1-qtx}\right)^m \\
        &= \sum_{k, \ell,n \geq 0} (qty)^k \left(\frac{xt}{1-qtx}\right)^{\ell} (y(t+1))^n {n+ k \choose k} {n+k + \ell \choose n+k} \frac{1}{\left(1-\frac{x}{1-qtx}\right)^{\ell + k + 1}} \\
        &= \frac{1-qtx}{1- qtx - x} \sum_{k, \ell, n \geq 0} \left(\frac{qty(1-qtx)}{1- qtx - x}\right)^k  (y(t+1))^n \left(\frac{xt}{1 - qtx -x}\right)^{\ell} {n+ k \choose k} {n+k + \ell \choose n+k} \\
        &= \frac{1-qtx}{1- qtx - x} g(q, t, x, y),
    \end{align*}
    where $g(q, t, x y)$ is defined to be 
    \[\sum_{k, \ell,n \geq 0} \left(\frac{qty(1-qtx)}{1- qtx - x}\right)^k  (y(t+1))^n \left(\frac{xt}{1 - qtx -x}\right)^{\ell} {n+ k \choose k} {n+k + \ell \choose n+k}.\]
    Then, notice that there is only one binomial coefficient involving $\ell$ in this expression, so we can eliminate the sum including $\ell$ as follows: 
    \begin{align*}
        &\phantom= g(q, t, x, y) \\
        &= \sum_{k, \ell, n \geq 0} \left(\frac{qty(1-qtx)}{1- qtx - x}\right)^k  (y(t+1))^n \left(\frac{xt}{1 - qtx -x}\right)^{\ell} {n+ k \choose k} {n+k + \ell \choose n+k} \\
        &= \sum_{k, n \geq 0} \left(\frac{qty(1-qtx)}{1- qtx - x}\right)^k (y(t+1))^n {n+ k \choose k} \sum_{\ell \geq 0} {n+k + \ell \choose n+k} \left(\frac{xt}{1 - qtx -x}\right)^{\ell} \\
        &= \sum_{k, n \geq 0} \left(\frac{qty(1-qtx)}{1- qtx - x}\right)^k (y(t+1))^n {n+ k \choose k} \frac{1}{\left(1-\frac{xt}{1 - qtx -x}\right)^{n+k+1}} \\
        &= \frac{1-qtx - x}{1 - qtx -x -xt} \sum_{k, n \geq 0} {n+ k \choose k} \left(\frac{y(t+1)(1-qtx - x)}{1 - qtx - x - xt}\right)^{n} \left(\frac{qty(1-qtx)}{1 - qtx - x - xt}\right)^{k} \\
        &= \frac{1-qtx - x}{1 - qtx -x -xt} h(q, t, x, y),
    \end{align*}
    where $h(q, t, x, y)$ is defined to be 
    \[\sum_{k, n \geq 0} {n+ k \choose k} \left(\frac{y(t+1)(1-qtx - x)}{1 - qtx - x - xt}\right)^{n} \left(\frac{qty(1-qtx)}{1 - qtx - x - xt}\right)^{k}.\]
    
    Then, we can similarly eliminate the sum involving $n,$ and also use the rational function form of an infinite geometric series, to get that $h(q, t, x, y)$ has the following rational function representation: 
    
    \begin{align*}
        &\phantom= h(q, t, x, y) \\
        &= \sum_{k, n \geq 0} {n+ k \choose k} \left(\frac{y(t+1)(1-qtx - x)}{1 - qtx - x - xt}\right)^{n} \left(\frac{qty(1-qtx)}{1 - qtx - x - xt}\right)^{k} \\
        &= \sum_{k \geq 0} \left(\frac{qty(1-qtx)}{1 - qtx - x - xt}\right)^{k} \sum_{n \geq 0}{n+ k \choose k} \left(\frac{y(t+1)(1-qtx - x)}{1 - qtx - x - xt}\right)^{n} \\
        &= \sum_{k \geq 0} \left(\frac{qty(1-qtx)}{1 - qtx - x - xt}\right)^{k} \frac{1}{\left(1 - \frac{y(t+1)(1-qtx - x)}{1 - qtx - x - xt}\right)^{k + 1}} \\
        &= \frac{1 - qtx - x - xt}{1 - qtx - x - xt - y(t+1)(1-qtx - x)} \sum_{k \geq 0} \left(\frac{qty(1-qtx)}{1 - qtx - x- xt - y(t+1)(1-qtx-x)}\right)^k \\
        &= \left(\frac{1 - qtx - x - xt}{1 - qtx - x - xt - y(t+1)(1-qtx - x)}\right)\left( \frac{1}{1 - \frac{qty(1-qtx)}{1 - qtx - x- xt - y(t+1)(1-qtx-x)}}\right) \\
        &= \frac{1 - qtx - x - xt}{1 - qtx - x- xt - y(t+1)(1-qtx-x) - qty(1-qtx)}
    \end{align*}
    Therefore, we have 
    \begin{align*}
        &\phantom=\phantom. \sum_{m,n \geq 0} M_{m, n}(-q, -t) x^m y^n \\
        &= \frac{1}{1-qtx} f(q, t, x, y) \\
        &= \left(\frac{1}{1 - qtx}\right)\left( \frac{1-qtx}{1- qtx - x}\right) g(q, t, x, y) \\
        &= \left(\frac{1}{1 - qtx - x}\right)\left( \frac{1-qtx - x}{1 - qtx -x -xt}\right) h(q, t, x, y) \\
        &= \left(\frac{1}{1-qtx - x - xt}\right)\left(\frac{1 - qtx - x - xt}{1 - qtx - x- xt - y(t+1)(1-qtx-x) - qty(1-qtx)}\right) \\
        &= \frac{1}{(1-x(qt - t+1))(1-y(qt - t+ 1)) + t(1-t)(q+1)xy},
    \end{align*}
    as desired. 
\end{proof}

Then, using Lemma~\ref{lem:lhs_rational_function}, we can find the rational function representation of the $M$-triangle, by setting $-q = q$ and $-t = t,$ which gives us 
\begin{equation}\label{eq:m-triangle-rational}
\sum_{m, n \geq 0} M_{m, n}(q, t) x^m y^n = \frac{1}{(1-x(qt - t+1))(1-y(qt - t+ 1)) + t(1-t)(q+1)xy},
\end{equation}
proving Theorem~\ref{thm:m-rational-function} as introduced at the beginning of the paper.

\end{document}